\documentclass[12pt]{article}
\usepackage[utf8]{inputenc}
\usepackage{amsmath, amssymb, amsthm, color, float}
\usepackage[english]{babel}
\usepackage{dirtytalk}
\usepackage{graphicx}
\usepackage{enumitem}
\newlist{Properties}{enumerate}{2} \setlist[Properties]{label=\textbf{Property} \arabic*.,itemindent=*}
\usepackage[noend]{algpseudocode}
\newcommand{\me}[1]{{\textcolor{magenta}{[}\textcolor{magenta}{{M.E.: \em #1}]}}}

\newtheorem{theorem}{Theorem}
\newtheorem{lemma}[theorem]{Lemma}
\newtheorem{remark}[theorem]{Remark}
\newtheorem{observation}[theorem]{Observation}

\newtheorem{conjecture}[theorem]{Conjecture}
\newtheorem{corollary}[theorem]{Corollary}
\newtheorem{proposition}[theorem]{Proposition}
\usepackage[letterpaper, portrait, margin = 1 in]{geometry}

\title{A note on hyperopic cops and robber}

\author{N.E. Clarke \thanks{Department of Mathematics and Statistics, Acadia University, Wolfville, NS, Canada} \and S. Finbow \thanks{Department of Mathematics and Statistics, St. Francis Xavier University, Antigonish, NS, Canada} \and M.E. Messinger \thanks{Department of Mathematics and Computer Science, Mount Allison University, Sackville, NB, Canada} \and A. Porter\thanks{Department of Mathematics and Computer Science, Mount Allison University, Sackville, NB, Canada}}
\date{\today}

\begin{document}
\maketitle

\begin{abstract}
We explore a variant of the game of Cops and Robber introduced by Bonato et al.~where the robber is invisible unless outside the common neighbourhood of the cops. The hyperopic cop number is analogous to the cop number and we investigate bounds on this quantity. We define a small common neighbourhood set and relate the minimum cardinality of this graph parameter to the hyperopic cop number. We consider diameter $2$ graphs, particularly the join of two graphs, as well as Cartesian products. 
\end{abstract}
\section{Introduction}

In the original game of Cops and Robber, both the ``cops'' and the ``robber'' have equal access to perfect information, as both the cops and the robber are visible throughout the entire game. This idealistic situation, however, does not fully encapsulate many pursuit and evasion scenarios. In recent years, there has been an increase in research that explores more practical scenarios in which the cops or robber have varied visibility (see~\cite{bonato, clarke, limitedvis,zerovis,zerovis2} for example). 

In \cite{bonato}, Bonato et al.~introduce an imperfect information variant of Cops and Robber, called ``Hyperopic Cops and Robber'', where the robber is visible \emph{unless} located on a vertex that is in the neighbourhood of every cop (i.e. if the robber is ``close'' to every cop, then the robber will be invisible).  As described in~\cite{bonato}, the motivation for this variant is a predator-prey system, where the prey has a short-range defense mechanism such as a squid releasing ink. Because these are short range defenses, when the cops are far away from the robber they will be unaffected by the prey's defense and continue to have perfect visibility.

More precisely, Hyperopic Cops and Robber is a pursuit-evasion game played on a reflexive graph. There are two players: one player who controls a finite set of cops and another player who controls a single robber. Initially, the cops choose a multiset of vertices to occupy and then the robber chooses a vertex to occupy. At each time step, a subset of cops move to an adjacent vertex, and the robber then moves to an adjacent vertex. As described above, the cops play with imperfect information: when the robber is in the neighbourhood of every cop (i.e. the vertices occupied by the cops are each adjacent to the vertex occupied by the robber), the robber is invisible.  In contrast, the robber plays with perfect visibility and can view the location of each cop at all times. Since each graph is reflexive, note that a cop or robber can \emph{pass} during a ``move'', which equates to traversing the incident loop.

If, after a finite number of moves, at least one cop occupies the same vertex as the robber, then we say the cops have \emph{captured} the robber. The cops' objective is to capture the robber, while the robber attempts to avoid this situation indefinitely. Note that by initially placing a cop on every vertex of the graph, the robber would be trivially captured. As a result, this paper focuses on finding the \emph{minimum} number of cops required to guarantee the capture of the robber, defined as the \emph{hyperopic cop number} of graph $G$ and denoted $c_H(G)$. The hyperopic cop number is the corresponding analogue to the cop number, $c(G)$, in the original game.  We note that $c(G) \leq c_H(G)$ as any strategy used by hyperopic cops to capture a robber could also be used by cops to win in the original game.  

In~\cite{bonato}, Bonato et al.~show that for a graph $G$ with diameter $3$ or greater, $c_H(G) \leq c(G)+2$, and they provide an improved bound when, in addition to the diameter requirement, $\delta(G)\leq c(G)$, namely $c_H(G) \leq c(G)+1$ (where $\delta(G)$ is the minimum degree in $G$). For graphs of diameter $2$, the relationship between these two parameters is increasingly more complex.  In~\cite{bonato}, the authors also show that the hyperopic cop number of a diameter 2 graph can be unbounded as a function of either the cop number or the order of the graph, and consider graph joins (defined below in the next paragraph), which are always of diameter at most 2. We further study Hyperopic Cops and Robber in diameter 2 graphs and the main results of this paper stem from an exploration of graph joins.   

We begin, however, by exploring the relationship between Hyperopic Cops and Robber and the original game by considering induced subgraphs and isometric paths in Section~\ref{relationship}. In Section~\ref{scn} we define a new graph parameter to provide bounds on our parameter of interest, particularly for graph joins.  Recall $G \vee J$ is the join of graphs $G$ and $J$ where $V(G \vee J)=V(G) \cup V(J)$ and $E(G \vee J) = E(G)\cup E(J) \cup \{uv~:~ u \in V(G) \text{ and } v \in V(J) \}$.  In Section~\ref{sec:graph joins}, we focus on diameter $2$ graphs and, in particular, graph joins to develop a general upper bound and exact results in some specific cases. Finally, in Section~\ref{sec:cartesian product}, we consider Cartesian products.

To avoid trivial cases, all graphs considered in this paper are finite, connected, and undirected unless otherwise indicated.

\section{Relationship to Cops and Robber}\label{relationship}
Restricting the visibility of the cops adds a layer of complexity that prohibits us from using or translating many results from the original game of Cops and Robber.

 Let $J$ be an induced subgraph of $G$ formed by deleting one vertex.  Then $J$ is a {\it retract} of $G$ if there is a homomorphism $f$ from $G$ onto $J$ so that $f(x)=x$, for $x \in V(J)$; recall that if $J$ is a retract of $G$, then $c(J) \leq c(G)$~\cite{BI}.  However, such a bound does not always hold for the game of Hyperopic Cops and Robber.  Let $G_n$ be the graph obtained by adding a leaf to one vertex of $K_n$; see Figure \ref{fig:induced subgraph} for reference.  Then $c_H(K_n) = \lceil \frac{n}{2}\rceil$~\cite{bonato}, but it is easy to see that $c_H(G_n) = 2$: one cop occupies the leaf and the second cop occupies any other vertex.  The robber will be visible unless the robber occupies the vertex adjacent to the leaf.  In either situation, the robber is captured immediately.  Thus $K_n$ is a retract of $G_n$ and $c_H(K_n) > c_H(G_n)$ for $n \geq 5$. 
 
 \begin{remark}  If graph $J$ is a retract of graph $G$, then $c(J) \leq c(G)$, but it is not necessarily true that $c_H(J) \leq c_H(G)$.\end{remark}

\begin{figure}[htbp]
    \centering
    \includegraphics[height= 4.5cm]{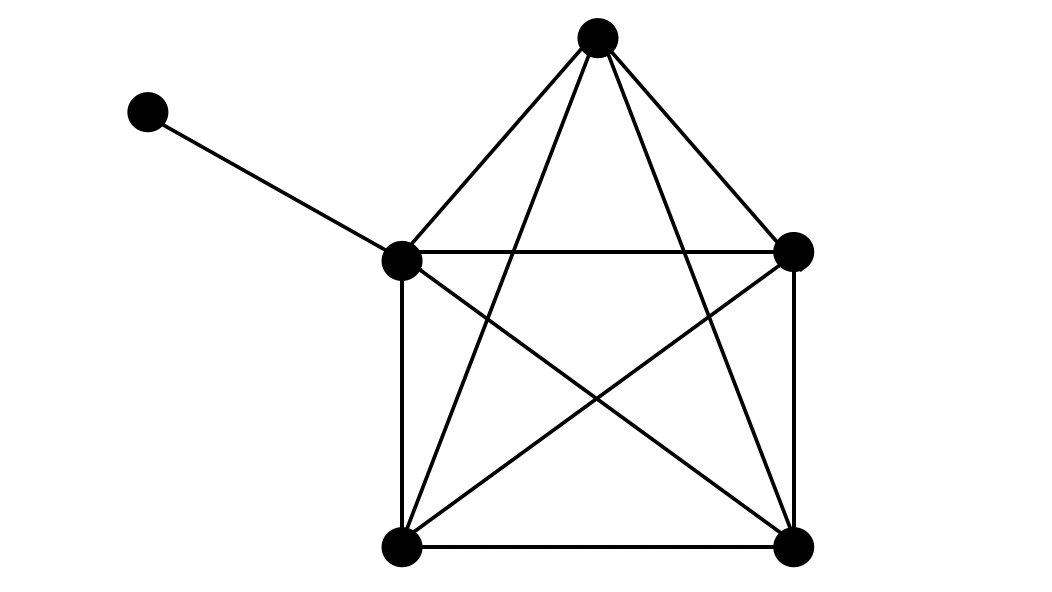}
    \caption{The graph $G_5$}
    \label{fig:induced subgraph}
\end{figure}

A path $P$ in $G$ is \emph{isometric} if, for all vertices $v$ and $w$ in $P$, $d_P(v,w)= d_G(v,w)$.  In~\cite{AF}, it is shown that in the original game of Cops and Robber, an isometric path is $1$-guardable. In other words, one cop can ensure that, after some time step, if the robber ever occupies a vertex of the isometric path, the robber will immediately be caught. This useful result has been used to prove many bounds on the cop number of a graph. 

\begin{remark} In the game of Cops and Robber, an isometric path is $1$-guardable, but in the game of Hyperopic Cops and Robber, an isometric path is not necessarily $1$-guardable. \end{remark}

Figure~\ref{fig:strongproduct} illustrates the strong product of $P_2$ and $P_3$ with an isometric path comprised of vertices $v_1$, $v_2$, $v_3$ depicted in bold. Observe that if the robber occupies (0,1) they will be invisible to the cop.  If the cop then occupies $v_2$, the robber will be invisible and can move to $v_1$ or $v_3$ unseen, and otherwise there is an unguarded path the robber can move onto.  

\begin{figure}[htbp]
    \centering
    \includegraphics[height= 4cm]{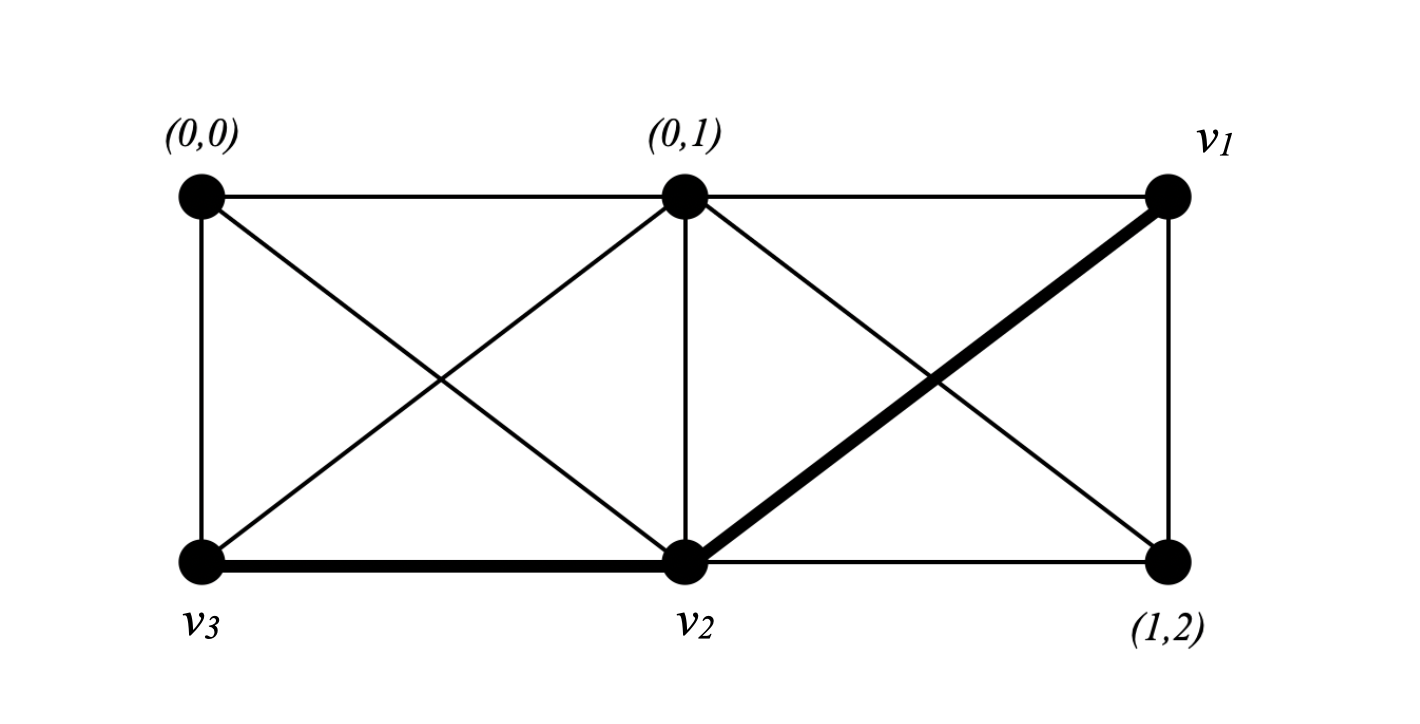}
    \caption{An isometric path in $P_2 \boxtimes P_3$ shown in bold}
    \label{fig:strongproduct}
\end{figure}

The following proof that an isometric path is $2$-guardable mirrors that of Theorem 1.7 in~\cite{bonato3}. 

\begin{theorem}
In the game of Hyperopic Cops and Robber, an isometric path is 2-guardable.
\end{theorem}
\begin{proof}  
Let $P = v_0,v_1,\,\dots,v_k$ be an isometric path in a graph $G$.  For $i \in \{0,1,\dots,k-1\}$, let $D_i = \{x \in V(G): d(x,v_0) = i\}$ and $D_k = \{x \in V(G): d(x,v_0) \geq k\}$. Since $P$ is an isometric path, it follows that $v_i \in D_i$, for $i \in \{0,1,\,\ldots, k\}$.

Two cops, restricted to $P$, play as if the robber is on $v_j$ when the robber is on some vertex of $D_j$, for $j \in \{ 0,1,\,\ldots,k\}$; i.e.~the cops play on the \emph{robber's image}. If the robber is in $D_j$, their only choice is to move to a vertex in $D_{j-1}, D_j,$ or $D_{j+1}$ and therefore their image can only be one of $v_{j-1},v_{j}$, or $v_{j+1}$.  (Note that if $j=k$, the robber can only move to a vertex in $D_{j-1}$ or $D_j$.)

If $k=1$, the cops occupy $v_0$ and $v_1$ and the robber can never occupy a vertex of $P$.  If $k=2$ then place a cop, say $C_1$, on $v_0$ and a cop, say $C_2$, on $v_1$. The robber can only enter $P$ by moving to $v_2$.  In that case, the robber will be seen immediately, as there is no edge between $v_0$ and $v_2$ so cop $C_2$ can then move to $v_2$ to capture the robber.  

If $k >2$, place cop $C_1$ on $v_0$ and cop $C_2$ on $v_3$. Note that $v_0$ and $v_3$ share no common neighbours since $P$ is isometric, and therefore the robber is initially visible. While the robber is in $D_i$ where $i \notin \{0,1,2\}$, cop $C_2$ can move to capture the robber's image. After cop $C_2$ captures the robber's image, if the robber moves to a vertex in $D_j$, $j \notin \{0,1,2\}$, cop $C_2$ moves to $v_j$.  So if the robber moves to $v_j$ for $j \in \{3,4,\dots,k\}$, cop $C_2$ will move to $v_j$ and capture the robber.

If at any time the robber is located or moves to  a vertex of $D_1 \cup D_2$, cop $C_2$ will be located at $v_3$. (Note that since cop $C_2$ follows the image of the robber if the robber is not in $D_1 \cup D_2$, we can assume that once the robber enters $D_2$, cop $C_2$ will be on $v_3$.) If the robber attempts to enter $P$ from a vertex of $D_2$ or $D_1$, then the robber will move to either $v_1$ or $v_2$. In either case, the robber will be visible to the cops and occupy a vertex adjacent to a cop; that cop will move to the vertex occupied by the robber and capture the robber.\end{proof}

\section{Small Common Neighbourhood Set}\label{scn}
We begin by defining a \emph{small common neighbourhood set} and later in the section, relate the minimum cardinally of such a set to the hyperopic cop number.  

\subsection{Definition and Properties}\label{sec:defining}

Let $G$ be a graph and note that $G$ may be disconnected. Define a non-empty set $S$ to be a \emph{small common neighbourhood set of $G$} if $$\left| \bigcap_{v \in S} N(v)\right|\leq |S|.$$ Let $\Upsilon (G)$ be the minimum cardinality of a \emph{small common neighbourhood set} of $G$. 

If we consider $S = V(G)$, then clearly such a set is a small common neighbourhood set of $G$; thus $\Upsilon(G) \leq |V(G)|$.  We additionally note that since $S$ is defined to be non-empty, we know $\Upsilon(G) > 0$ and hence, the parameter $\Upsilon$ is well-defined.

To further explain this new graph parameter, below is a pseudo algorithm that will determine $\Upsilon(G)$ for any finite graph $G$.\medskip

Let $\mathcal{P}(V(G))$ be the power-set of all vertices in $G$. 
\begin{algorithmic}[1]

    \State Set $\Upsilon(G) = 0$ and set $i = 1$. Start by labelling every set in $\mathcal{P}(V(G))$ as unchecked
    \While{$\Upsilon(G) = 0$}
        \While{ $\exists \ Z \in \mathcal{P}(V(G))$ that is unchecked and $|Z| = i$}
            \If{$|\bigcap_{v \in Z} N(v)| \leq i$}
                 \State Set $\Upsilon(G) = i$. 
                 \State Identify $Z$ as a minimum small common neighbourhood set. 
            \Else
                \State Label the set $Z$ as checked. 
        \EndIf
        \EndWhile  
    \State i = i +1
\EndWhile

\end{algorithmic} 

Clearly, the value produced by the algorithm is the minimum cardinality of a small common neighbourhood set of $G$.  

\begin{proposition}
\label{Upsilon Property}
For any graphs $G$ and $J$, $\Upsilon(G \vee J) \leq \Upsilon(G) + \Upsilon(J)$.
\end{proposition}

\begin{proof}
Let $S_G$, $S_J$ be minimum small common neighbourhood sets of $G$ and $J$, respectively. Consider the set $S=S_G\cup S_J$.  Clearly, $$\bigcap_{v \in S} N_{G\vee J}(v)= \bigcap_{v \in S_G} N_G(v)\cup \bigcap_{v \in S_J} N_J(v).$$ It follows that $S$ is a  small common neighbourhood set of $G\vee J$. Hence $\Upsilon(G \vee J) \le \Upsilon(G) + \Upsilon(J)$.\end{proof}

Though the bound in Proposition~\ref{Upsilon Property} is obviously tight for some graphs, it is not always the case. Let $H$ be the strong product of $P_4$ and $P_2$; see Figure \ref{fig: two copies of H} for labeled vertices. Clearly $\Upsilon(H) = 2$ as the set $\{v,w\}$ forms a minimum small common neighbourhood set.  Now consider the graph $H \vee H$; here the set $\{v,w,z\}$ forms a minimum small common neighbourhood set (as indicated in Figure \ref{fig: two copies of H}), so clearly $\Upsilon(H \vee H)< \Upsilon(H)+\Upsilon(H)$.

\begin{figure}[ht]
    \centering
    \includegraphics[height = 5 cm]{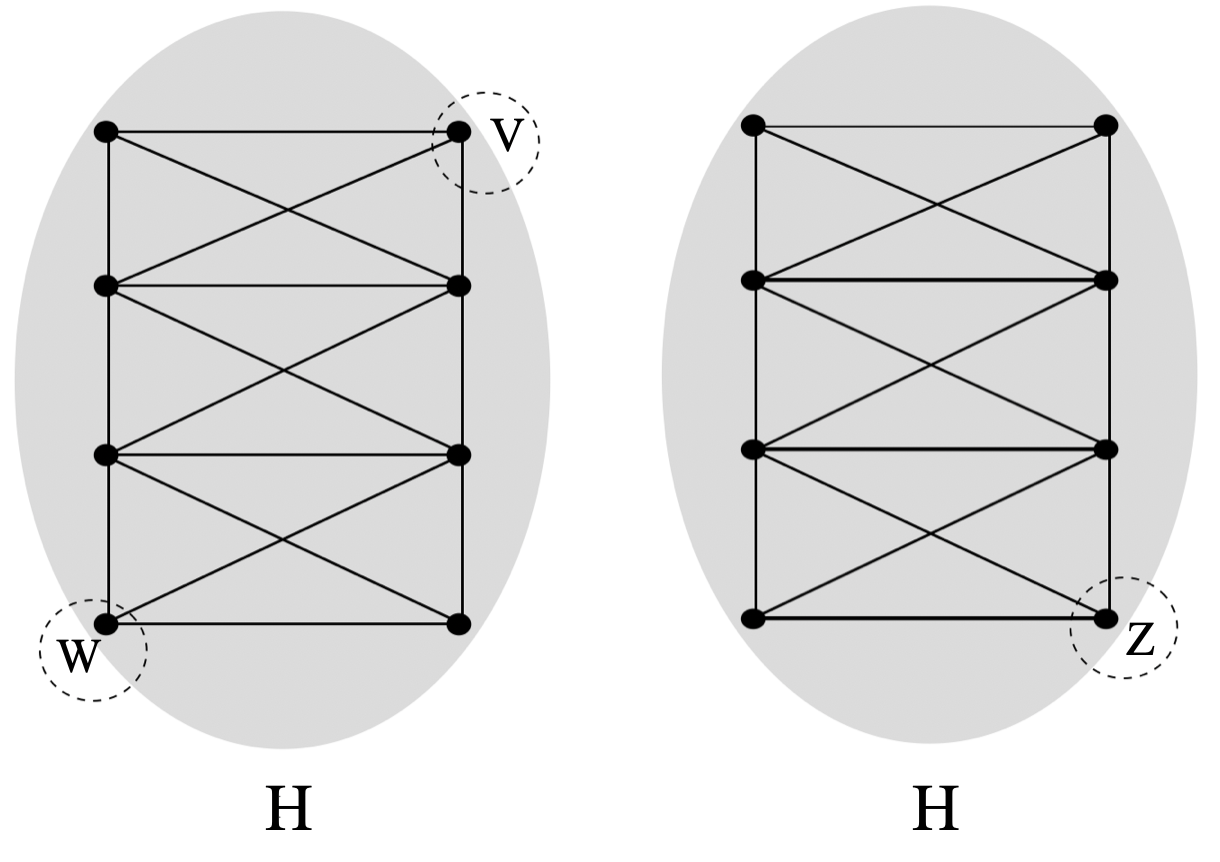}
    \caption{Two copies of graph $H$.}
    \label{fig: two copies of H}
\end{figure}

\subsection{Connections to Hyperopic Cops and Robber}\label{sec:resulting}

We next use the parameter $\Upsilon$ to provide a general upper bound for any graph. First, we state an useful observation that identifies the position of the robber if they are invisible. 

\begin{observation}
\label{upsilon invisible lemma}
Let $G$ be a graph and $S$ be a  set of $G$. If there is a cop located on every vertex in $S$ and the robber is invisible, then the robber must be located on some vertex in $\bigcap_{v \in S} N(v) $.
\end{observation}

\begin{theorem}\label{Upsilon all graph}
For any graph $G$, $c_H(G) \leq c(G) + \Upsilon(G)$. 
\end{theorem}

\begin{proof} Let $G$ be a graph and $S$ be a small common neighbourhood set of minimum cardinality. Place a cop on each vertex of $S$ and place the remaining $c(G)$ cops arbitrarily on vertices of $G$.

If the robber is invisible, then the robber must be in $\cap_{v \in S} N(v)$, as per Observation \ref{upsilon invisible lemma}. The $\Upsilon(G)$ cops on set $S$ can immediately move to $\cap_{v \in S} N(v)$ to capture the robber since $|\cap_{v \in S} N(v)| \leq \Upsilon(G)$ by definition. 

If the robber is visible, then the $\Upsilon(G)$ cops located on set $S$ remain on the vertices of $S$ (to ensure the robber remains visible) while the remaining $c(G)$ cops follow the winning strategy from the original game of Cops and Robber in order to capture the robber. \end{proof}

The following results all appear in~\cite{bonato}; however, these results can all be proven in a simple and direct way by utilizing the concept of a  small common neighbourhood. Consequently, we include the proofs below as an illustration of the usefulness of the parameter.

\begin{theorem}\label{thm:bonato}
Let $G$ be a graph.

(1) If $diam(G) \geq 3$ then $c_H(G) \leq c(G)+2$. 

(2) If $G$ contains a cut-vertex, then $c_H(G) \leq c(G)+1$. 

(3) If $G$ is triangle free, then $c_H(G) \leq c(G)+1$.
\end{theorem}

\begin{proof} 

For {\it(1)}, let $G$ be a graph with $diam(G) \geq 3$. Then there exist vertices $v, w \in V(G)$ such that $d(v,w) \geq 3$. As $|N(v) \cap N(w)| = 0$, the set $\{v,w\}$ forms a small common neighbourhood set and therefore $\Upsilon(G) \leq 2$. By Theorem \ref{Upsilon all graph}, $c_H(G) \leq c(G) +2$.\medskip

For {\it(2)}, Let $G$ be a graph with a cut-vertex $u$ with non-adjacent neighbours $v$, $w$.  Further, let $S_v$ and $S_w$ be sets such that $S_v \cup \{u\} \cup S_w = V(G)$ and $v,w$ are in different components of the induced subgraph $G[S_v \cup S_w]$. Initially, place one cop on $v$, one cop on $w$, and $c(G)-1$ cops on $u$.  Since $\{u,v,w\}$ is a small common neighbourhood set ($N(u) \cap N(v) \cap N(w) = \emptyset$), the robber is visible. Without loss of generality, suppose the robber is on a vertex in $S_w$. The cop located at $v$ remains at vertex $v$ while the robber is visible (and in $S_w$), while the remaining $c(G)$ cops move on the vertices of the subgraph $G[S_w \cup \{u\}]$ to follow a winning strategy from the original game of Cops and Robber and capture the robber. If, at any point the robber becomes invisible, it must be because the robber has moved to $u$.  In this case, the cop at $v$ moves to $u$ to capture the robber.\medskip  

For {\it (3)}, let $G$ be a triangle-free graph with no cut-vertex and consider the power set $\mathcal{P}(V(G))$. Since $G$ is triangle-free, for all sets $X \in \mathcal{P}(V(G))$ where $|X| = 2$, $X$ satisfies the conditions for a small common neighbourhood set on $G$. This is because for any adjacent vertices $v,w \in X$, $|N_G(v) \cap N_G(w)| = 0 \leq |X| $. We place $c(G)+1$ cops arbitrarily on vertices of $G$.  Suppose cop $C_1$ is located on vertex $x$.   If $C_1$ remains on $x$, then no matter the locations of the other $c(G)$ cops, the robber cannot be in the common neighbourhood of both $x$ and all the vertices occupied by the other $c(G)$ cops. Therefore we can conclude that the robber will be visible throughout the game; $C_1$ remains at $x$ and the other $c(G)$ cops will follow a winning strategy for the original game of Cops and Robber.\end{proof}

\begin{corollary}
    \label{diam 2, universal vertex with degree 3 vertex}
Let $G$ be a graph with a universal vertex.  If  there exists $ v \in V(G)$ such that  $deg(v) \leq 3$, then $c_H(G) \leq 2$.\end{corollary}

\begin{proof}Let $u$ be the universal vertex of $G$ and $v$ be a vertex with $deg(v) \leq 3$. Note that the set $\{v,u\}$ forms a small common neighbourhood set because $|N(v) \cap N(u)| = 2 \leq \{v,u\}$. Therefore, let $S = \{v,u\}$.  Start by placing one cop on $u$ and the other cop on $v$.  If the robber is visible, then the cop on $u$ can easily move to the robber's position to capture. If the robber is invisible, then they must be in $N(v) \cap N(u)$, as per Observation~\ref{upsilon invisible lemma}.  Since $S$ is a small common neighbourhood set, then the cops on $u$ and $v$ can move to capture the robber in $N(v) \cap N(u)$.\end{proof}

\section{Diameter 2: Graph Joins}\label{sec:graph joins}

In \cite{bonato}, the authors determine the hyperopic cop number for diameter $1$ graphs  (i.e.~complete graphs). Furthermore, they  bound the hyperopic cop number in terms of the original cop number for graphs of diameter $3$ or greater. Little is known, however, about the hyperopic cop number of diameter $2$ graphs, which motivates the work in this section. These dense graphs provide many opportunities for the robber to  be invisible initially and throughout the game. From $\cite{bonato}$, we can see that when a graph has diameter $1$, the hyperopic cop number ``blows up'': $c_H(K_n)= \lceil \frac{n}{2}\rceil$. It was additionally shown in \cite{bonato} that some graphs of diameter $2$ also have large hyperopic cop numbers. Consider the graph $K_n - e$, which is a complete graph with one edge removed. From \cite{bonato}, $diam (K_n - e) = 2$ and $c_H (K_n - e) = \lfloor \frac{n}{2} \rfloor$.  In this section, we examine the hyperopic cop number for a large family of diameter $2$ graphs: graph joins. 

 Note that the join $G \vee J$ is connected, regardless of whether $G$ or $J$ are connected themselves. Furthermore, the join of two graphs always has $diam(G \vee J) \leq 2$. As a result, $G \vee J$ will have diameter $2$, unless both $G$ and $J$ are complete graphs, in which case $diam(G \vee J) = 1$.

Furthermore, we note that for two graphs $G$ and $J$, $\gamma(G \vee J) \leq 2$, where $\gamma$ denotes the domination number. From this, it is clear that in the original game of Cops and Robber, the cop number of a graph join $G \vee J$ is always at most $2$: place one cop on a vertex of $G$ and the other cop on a vertex of $J$. If we place a hyperopic cop on a vertex $u \in V(G)$ and a vertex $v \in V(J)$, then $|N_{G \vee J}(u) \cap N_{G\vee J}(v)| = \deg_G(u)+\deg_J(v)$ and a robber located on one of the $\deg_G(u)+\deg_J(v)$ vertices in the common neighbourhood would be invisible and therefore in most cases would not immediately be captured.

Consider two complete graphs $K_m$ and $K_n$. From the definition of graph joins, we know that $K_m \vee K_n = K_{m+n}$ and $c_H(K_{m+n}) = \lceil \frac{m+n}{2} \rceil$ and $c(K_{m+n})=1$.  This example illustrates how $c_H(G \vee J)$ may be very different from $c(G \vee J)$ and that the hyperopic cop number of the join of two graphs acts independently from the hyperopic cop number of its two factors.  We begin with a useful observation.

\begin{observation}
\label{lemmajoin}
Let $G$ and $J$ be graphs and consider the game of Hyperopic Cops and Robber played on $G \vee J$. If the robber becomes visible during any round and there is at least one cop in each of $G$ and $J$, then the robber will be captured during the cops' next move. 
\end{observation}

Note that $\gamma (G \vee J) \leq 2$ and $diam(G \vee J) \leq 2$. If $\gamma(G \vee J) =1$ then both $G \cong K_n$ and $J \cong K_m$ for some $m,n \geq 1$. Therefore, the robber will never be visible. Assume then $\gamma(G \vee J) = 2$ and, without loss of generality, the robber is in $G$. If the robber becomes visible and there is a cop in $J$, then the cop can move to $G$ to capture the robber.

\begin{theorem}
\label{bothupsilon}
For any graphs $G$ and $J$, $c_H(G \vee J) \leq \Upsilon(G) + \Upsilon(J)$. 
\end{theorem}
 
\begin{proof}
Let $S_G$ and $S_J$ be minimum small common neighbourhood sets of $G$ and $J$, respectively. Note that $|S_G| = \Upsilon (G)$ and $|S_J| = \Upsilon(J)$ by definition. Start by placing $\Upsilon(G)$ cops on all vertices in $S_G$ and $\Upsilon(J)$ cops on all vertices in $S_J$. If the robber is visible then, by Observation~\ref{lemmajoin}, they will be caught on the cops' next move. Therefore, we can assume the robber is invisible and thus located in $\bigcap_{w \in S_G} N(w)$ or $\cap_{v \in S_J} N(v)$. Since $|\cap_{w \in S_G} N(w)| \leq |S_G| = \Upsilon(G)$ and $|\cap_{v \in S} N(v)| \leq |S_J| = \Upsilon(J)$, the $\Upsilon(G)$ cops on $S_G$ can move to $\cap_{w \in S_G} N(w)$ and the $\Upsilon(J)$ cops on $S_J$ can move to $\cap_{v \in S_J} N(v)$ to capture the robber.\end{proof}

Recall from Proposition~\ref{Upsilon Property} that for some graphs $G$ and $J$, $\Upsilon(G \vee J) \leq \Upsilon(G) + \Upsilon(J)$. This property motivates the following theorems which, for some graphs, improve upon Theorem \ref{bothupsilon}.

\begin{theorem}
Let $G$ and $J$ be graphs. If there exists a set $S$ such that $S \not\subseteq V(G)$, $S \not\subseteq V(J)$, and $S$ is a minimum small common neighbourhood set of $G \vee J$ then $c_H(G \vee J) \leq \Upsilon(G \vee J)$. \end{theorem}

\begin{proof}
Let $S$ be a minimum small common neighbourhood set of $G \vee J$ where $S \not\subseteq V(G)$ and $S \not\subseteq V(J)$. Start by placing a cop on each vertex in $S$. If the robber is visible (since $S$ has vertices in both $G$ and $J$), by Lemma \ref{lemmajoin} the robber will be caught. Therefore, we can assume that the robber is invisible from the start. By Observation~\ref{upsilon invisible lemma} the robber is located in $\cap_{v \in S} N(v)$ and as $|\cap_{v \in S} N(v)| \leq \Upsilon(G \vee J)$, the cops can move to $\cap_{v \in S} N(v)$ to capture the robber.\end{proof}

\begin{theorem}
Let $G$ and $J$ be graphs. Let $\mathcal{S}$ be the set of all minimum small common neighbourhood sets. If for all sets $S \in \mathcal{S}$,  $S \subseteq V(G)$ or $S \subseteq V(J)$ then $c_H(G \vee J) \leq \Upsilon (G \vee J) + 1$.\end{theorem}

\begin{proof}
Let $S$ be a minimum small common neighbourhood set of graph $G \vee J$. Without loss of generality, assume that $S \subseteq V(G)$. Start by placing cops on all vertices in $S$ and one cop on some vertex $w \in V(J)$. If the robber is initially visible, they will be caught by Observation~\ref{lemmajoin} (as there is a cop in both $G$ and $J$). If the robber is invisible, by Observation~\ref{upsilon invisible lemma}, we know that the vertex occupied by the robber is in the set $\bigcap_{v \in S}( N(v) \cap N(w)) = \bigcap_{v \in S\cup \{w\}} N(v)$. 

Since $|\bigcap_{v \in S} N(v)| \leq \Upsilon(G \vee J)$, the $\Upsilon(G \vee J)$ cops in $S$ can move to $\bigcap_{v \in S} N(v)$ to capture the robber.\end{proof}

Note that the join of two disconnected graphs will result in a connected graph. Thus, we can consider $G$ and $J$ to be disconnected.  

\begin{corollary}\label{cor:disc}
   Let $G$ and $J$ be  disconnected graphs with $n$ components and  $m$ components respectively. If $n,m \geq 2$, then $c_H(G \vee J) \leq 4$. 
\end{corollary}
\begin{proof}
Let $G_i$ be a component of $G$, for $i \in \{1,2,\,\ldots,n\}$ and $J_j$ a component of $J$, for $j \in \{1,2,\,\ldots,m\}$. 

Clearly, for some vertex $v_1 \in G_1$ and for some vertex $v_{2} \in G_{2}$, $N_G({v_1) \cap N_G(v_{2})} = \emptyset$ as $v_1$ and $v_{2}$ are in different components of $G$. Similarly, for some vertex $w_1 \in V(J_1)$ and for some vertex $w_{2} \in V(J_{2})$, $N_J({w_1) \cap N_J(w_{2})} = \emptyset$ as $w_1$ and $w_{2}$ are in different components of $J$. Therefore $\{v_1,v_{2}\} \cup \{w_1,w_{2}\}$ form small common neighbourhood sets for $G$ and $J$, respectively. By Theorem \ref{bothupsilon} we can see that $c_H(G \vee J) \leq 4$.\end{proof}

With respect to Corollary~\ref{cor:disc}, if a component of either $G$ or $J$ contains a vertex $z$ of degree $1$ or an isolated vertex, then $c_H(G \vee J) \leq 3 $ as $\{z\}$ would form a small common neighbourhood set.

\begin{theorem}
Let $G$ be a connected graph, and $J$ be a disconnected graph with $n$ components, $n \geq 2$. Then $c_H(G \vee J ) \leq c_H(G) +2$.
\end{theorem}
\begin{proof}
Let $J_i$ be a component of $J$, for $i \in \{1,2,\,\ldots,n\}.$ In $G \vee J$, we place one cop on a vertex $v_1 \in V(J_1)$, one cop on a vertex $v_{2} \in V(J_{2})$, and we place the remaining $c_H(G)$ on vertices of $G$. If the robber moves to a vertex of $J$, the robber will be visible  as $N_J({v_i) \cap N_J(v_{i+1})} = \emptyset$. The robber can then be caught immediately because there are cops in both $G$ and $J$, satisfying Observation \ref{lemmajoin}. Thus the robber must remain on vertices of $G$ where the $c_H(G)$ cops can carry out their winning strategy and catch the robber.\end{proof}

\section{Graph Products: Cartesian Product}
\label{sec:cartesian product}

Graph joins enable us to construct a new graph based on two input graphs.  Cartesian products also allow us to construct new graphs based on input graphs.  In this section, we bound the hyperopic cop number of the Cartesian product of graphs.

Let \(G\) and $J$ be connected, finite graphs with  \(V(G) = \{x_0,x_1,\,\ldots,x_{m}\}\) and \(V(J) = \{ y_0,y_1,\,\ldots,y_{n}\}\). The \emph{Cartesian product} of \(G\) and \(J\), denoted \( G \square J\), is the graph with \(V(G) \times V(J)\) as its vertex set. An edge $e \in E(G \square J)$ if and only if \(e =(x_i,y_j)(x_k,y_l)\) where either: \( i = k\) and \(y_j y_l \in E(J)\) or \( j = l\) and \(x_ix_k \in E(G)\).

Note that, unlike graph joins, Cartesian products do not guarantee a graph of diameter 2 or less. Instead, $diam(G \square J) = diam(G) + diam(J)$.  If $diam(G) \geq 2$ or $diam(J) \geq 2$, clearly $diam(G \square J) \geq 3$. This property allows us to exploit Theorem~6 from~\cite{bonato}, which states that if $diam(G) \geq 3$ then $c_H(G)\leq c(G)+2$.

\begin{corollary}
\label{Cartesian diameter 3}
   Let $G$ and $J$ be graphs where $diam(G) \geq 2$ or $diam(J) \geq 2$. Then $c_H(G \square J) \leq c(G)+ c(J) +2$.
\end{corollary}

If we instead consider the situation where $diam(G) \not\geq 2$ and $diam(J) \not\geq 2$, then it must be that $diam(G)=diam(J)=1$, which implies $G$ and $J$ are complete graphs.  We consider this situation next.

\begin{theorem}
\label{cartesian complete graphs}
If $n,m \geq 2$ then $c_H(K_n \square K_m) \leq 4$. 
\end{theorem}

\begin{proof}
Label the vertices of $K_n$ and $K_m$ as $v_1, v_2, \ldots v_n$ and $w_1, w_2, \ldots w_m$, respectively.

Clearly, $N((v_1, w_1)) \cap N((v_2, w_2)) = \{(v_2, w_1),(v_1, w_2)\}$. Therefore $\{(v_1, w_1), (v_1, w_2)\}$ forms a small common neighbourhood set and thus $\Upsilon(K_n \square K_m) = 2$. By Theorem \ref{Upsilon all graph}, we can see that $c_H(K_n \square K_m) \leq c(K_n \square K_m) + 2$. Furthermore, from \cite{tosic}, we know that $c(K_n\square K_m) \leq c(K_n) + c(K_m) = 2$. Thus, $c_H(K_n \square K_m) \leq 4$. \end{proof}

 We next improve the bound for graphs where $c(G) = c_H(G)$. The following proof mirrors that of Theorem $1$ in~\cite{tosic} which proves $c(G \square J) \leq c(G) + c(J)$.

\begin{theorem}
\label{graph and graph}
If \(G\) and \(J\) are finite connected graphs then \( c_H(G \times J) \leq c_H(G) + c_H(J) \).\end{theorem}

\begin{proof}\label{general cartesian} Let $G$ and $J$ be graphs and $V(G)=\{u_1,u_2,\dots,u_m\}$ and $V(J)=\{v_1,v_2,\dots,v_n\}$. For $i \in [n]$, let $G_i$ denote the induced subgraph of $G \square J$ with vertex set $V(G_i) = \{(u_1,v_i)$, $(u_2,v_i)$, $\dots,(u_m,v_i)\}$ and, for $k \in [m]$, let $J_k$ denote the induced subgraph of $G \square J$ with vertex set $V(J_k) = \{(u_k,v_1),(u_k,v_2),\dots,(u_k,v_n)\}$.  Further observe that $G_i \cong G$ for all $i \in [n]$ and $J_k \cong J$ for all $k \in [m]$.  In describing how the robber will be captured, the cops will move in two phases.\medskip

\underline{Phase 1:} Initially, $c_H(G)$ cops are placed on vertices of subgraph $G_1$ and one cop is placed on a vertex of $G_i$, for $2 \leq i \leq c_H(J)+1$.  

The $c_H(G)$ cops in $G_1$ follow a hyperopic winning strategy on $G_1$ in order to capture the first coordinate of the robber; that is, although the robber plays on vertices of $G \square J$, the cops will restrict their movements to subgraph $G_1$ and follow a winning strategy on $G_1$.  Once a cop, $g_1$ moves to occupy the same first coordinate of the robber, we say the cops on $G_1$ have captured the first coordinate of the robber.  Cop $g_1$ remains on the vertices of $G_1$ for the remainder of phase 1 and, whenever the robber changes his first coordinate, cop $g_1$ will change his first coordinate to match.  (Note that $g_1$ can maintain capture of the first coordinate because there can always be at least one cop on the graph that is not adjacent to the robber, so $g_1$ will always know where the robber is.)

Inductively, once cop $g_i$ has captured the first coordinate of the robber, the other $c_H(G)-1$ cops on $G_i$ will move to $G_{i+1}$.  Now there are $c_H(G)$ cops on $G_{i+1}$.  Similarly to the previous situation, the $c_H(G)$ cops on $G_{i+1}$ follow a hyperopic winning strategy until a cop $g_{i+1}$ moves to occupy the same first coordinate of the robber.  Cop $g_{i+1}$ remains on the vertices of $G_{i+1}$ for the remainder of phase 1 and whenever the robber changes their first coordinate, cop $g_{i+1}$ will change their first coordinate to match.

Once cops $g_1,g_2,\dots,g_{c_H(J)}$ have captured (and maintained) the first coordinate of the robber, we move to Phase 2.  \medskip

\underline{Phase 2:} For $i \in \{1,2,\dots,c_H(H)\}$, cop $g_i$ has the same first coordinate as the robber just before the robber's turn.  In this phase, whenever the robber changes their first coordinate, the cops $g_1,g_2,\dots,g_{c_H(J)}$ change their first coordinate to match the robber's.  Whenever the robber changes his second coordinate, the cops $g_1,g_2,\dots,g_{c_H(J)}$ move according to a hyperopic winning strategy.  More precisely, suppose the robber moves from $(u_k,v_\ell)$ to $(u_k,v_p)$.  We observe that the robber is located in $J_k$ before and after they move.  And the cops are also located in subgraph $J_k$.  Thus, they can move according to a hyperopic winning strategy on $J$.

The only issue is if the robber continually changes their first coordinate.  In this case, because there are an additional $c_H(G)$ cops that are not currently in use (i.e.~the cops other than $g_1,g_2,\dots,g_{c_H(J)}$), these cops will prevent the robber from continually changing their first coordinate.  To see this, suppose the robber continually changed their first coordinate.  Then their second coordinate, say $v_q$, would remain constant.  The $c_H(G)$ currently unused cops would move to subgraph $G_q$ and apply a hyperopic winning strategy on $G_q$; this would force the robber to either be captured or change their second coordinate.\end{proof}

From \cite{bonato} we know that $c_H(K_n) = \lceil\frac{n}{2} \rceil$ and, since $P_m$ is a tree, that $c_H(P_m) = 1$. Therefore, by Theorem \ref{graph and graph}, we can see that $c_H(K_n \square P_m) \leq \lceil\frac{n}{2} \rceil + 1$. In Theorem~\ref{thm:kP}, we determine $c_H(K_n \square P_m)$, which illustrates where Theorem \ref{graph and graph} provides too large of an upper bound on $c_H(K_n \square P_m)$. First, we state a useful lemma.

\begin{lemma}
\label{complete times p2}
Let $K_n$ be a complete graph where $|V(K_n)| = n$ and $P_2$ be a path with $2$ vertices. Two cops are sufficient to capture a robber in $K_n \square P_2$ in at most one move. 
\end{lemma}

\begin{proof}
Let $V(K_n) = \{k_1, k_2,\,\ldots, k_n\}$ and $V(P_2) = \{p_1, p_2\}$. For $i \in \{1,2\}$, let $K_n^i$ denote the induced subgraph of $K_n \square P_2$ with vertex set $V(K_n^i) = \{(k_1, p_j), (k_2, p_j),\,\ldots, (k_n, p_j)\}$, for $j \in [n]$. If $n =1$ or $n=2$, the cop's strategy is straightforward; we will consider $n \geq 3$. 

Start by placing one cop on vertex $(k_u, p_1)$, where $u \in [n]$, and the other cop on vertex $(k_v, p_2)$, where $v \in [n]$ and $ v > u$. Notice that $(k_u, p_1)$ and $(k_v, p_2)$ share only $2$ common neighbours $(k_u, p_2)$ and $(k_v, p_1)$. If the robber is invisible, the cops can infer the robber is situated on either of the common neighbours and move to  $(k_u, p_2)$ and $(k_v, p_1)$ to catch the robber in one move. 
 
If the robber is visible then they are either in $K_n^1$ or $K_n^2$. If the robber is in $K_n^1$ the cop on $(k_u, p_1) $ can move to capture the robber. Otherwise if the robber is in $K_n^2$, the robber on $(k_v, p_2)$ can move to capture the robber in one move. 
\end{proof}

\begin{theorem}\label{thm:kP}
For a complete graph $K_n$ with $|V(K_n)| = n$ and a path $P_m$ with $|V(P_m)| = m$, then $c_H(K_n \square P_m) = 2.$ 
\end{theorem}
\begin{proof}
From \cite{bonato}, we know $c_H(G) = 1$ if and only if $G$ is a tree. Since $K_n \square P_m$ contain cycles, $c_H(K_n \square P_m) \geq 2$.

Let $V(K_n) = \{k_1, k_2,\,\ldots, k_n\}$ and $V(P_m) = \{p_1, p_2,\, \ldots, p_m\}$. For $i \in [m]$, let $K_n^i$ denote the induced subgraph of $K_n \square P_m$ with vertex set $V(K_n^i) = \{(k_1, p_j), (k_2, p_j),\,\ldots, (k_n, p_j)\}$, for $j \in [n]$.

Start by placing one cop on $(k_u,p_1)$,  where $ u \in [n]$, and the other on $(k_v, p_m)$, where $v \in [n], v < u$. Notice  $(k_u,p_1) \in K_n^1$ and   $(k_v, p_m) \in K_n^m$. Since the cops do not share any common vertices at this point, we can assume that the robber is visible. Additionally, the robber will be restricted to the induced subgraph  $K_n \square P_m \setminus K_n^1 \cup K_n^m $ because if the robber is placed on $K_n^1$ or $K_n^m$ the robber will be seen and caught by the cop on $(k_u,p_1)$ or $(k_v, p_m)$, respectively. 

We will proceed using the following strategy: leave one cop on $(k_u,p_1)$ to guard $K_n^1$ and move the cop on $(k_v, p_m)$ along the path $P = (k_v, p_m),(k_v, p_{m-1}),\,\ldots, (k_v, p_2)$. Each time the cops access a new $K_m^i$, the robber is restricted to the induced subgraph $K_n \square P_m \setminus  \bigcup_{x=i}^{m} K_n^{x} \cup K_n ^1$. 

As the induced subgraph that the robber can be situated on decreases, they will eventually be on $K_n^2$. Once the robber travelling down $P$ is on $(k_v, p_2)$, the cops will be able to catch the robber during their next move using the strategy outlined in Theorem \ref{complete times p2}. 
\end{proof}

\section{Open Questions}

There still exist a variety of open problems associated with the hyperopic cop number. Most notably, the following conjectures involving small common neighbourhood sets remain unproven. 

\begin{conjecture}
For any connected graphs $G$ and $J$, $$c_H(G \vee J) \leq \min\{c_H(G), \Upsilon(G)\} + \min\{c_H(J), \Upsilon(J)\}.$$ 
\end{conjecture}

We have seen $c_H(G \vee J) \leq c(G \vee J)+\Upsilon(G \vee J)$ and $c_H(G \vee J) \leq \Upsilon(G) + \Upsilon(J)$.  The proof of this upper bound  provides an interesting challenge as the robber has access to \emph{any} vertex on the other graph, which may interfere with $c_H(G)$ cops and $c_H(J)$ cops carrying out their respective winning strategies. 

\begin{conjecture}
\label{conj: upsilon}
For connected graphs $G$ and $J$, $c_H(G \vee J) \leq \Upsilon(G \vee J)$.
\end{conjecture}
Although this would only improve the bound by one, it is a natural extension from the previous work done with small common neighbourhood sets. Since $\Upsilon(G \vee J)$ does not guarantee that the smallest common neighbourhood set contains vertices in both $G$ and $J$, we cannot utilize Observation \ref{lemmajoin}.  As a result, it is not simple to make the logical jump from Theorem~\ref{bothupsilon} to Conjecture~\ref{conj: upsilon}. Furthermore, it would also serve useful to determine which graph classes have the property that $\Upsilon(G \vee J) = \Upsilon(G) + \Upsilon(J)$. This would further narrow the question of how frequently there is a 
discrepancy between Theorem~\ref{bothupsilon} and Conjecture~\ref{conj: upsilon}. 

Determining a lower bound using a small common neighbourhood set would also be useful, ideally, a general lower bound that would accompany that given by Theorem~\ref{Upsilon all graph}. Another natural extension of this work would be results on other graph products such as the lexicographic and strong products. 

\section{Acknowledgements} 
The authors acknowledge research support provided by the Maple League of Universities (2020).


\begin{thebibliography}{9}

\bibitem{AF} M. Aigner, M. Fromme, A game of cops and robbers, Disc. Appl. Math. 8 (1984) 1--12.

\bibitem{BI} A. Berarducci, B. Intrigila, On the cop number of a graph, Adv. in Appl. Math. 14 (1993) 389–403. 

\bibitem{bonato} A.\ Bonato, N.E.\ Clarke, D.\ Cox, S.\ Finbow, F.\ Mc Inerney, M.E.\ Messinger, Hyperopic cops and robbers, Theoret. Comp. Sci. 794 (2018) 59--68.

\bibitem{bonato3}  A.\ Bonato, R.\ Nowakowski, \emph{The Game of Cops and Robbers on Graphs}, American Mathematical Society, Providence, RI, 2011.

\bibitem{clarke} N.E.\ Clarke, R. Nowakowski, Cops, Robber, and Photo Radar, Ars Comb. 56 (2000) 97--103.

\bibitem{limitedvis} D. Cox, N. Clarke, D. Dyer, S. Fitzpatrick, M. Messinger, Limited visibility cops and robber, Disc. Appl. Math. 282 (2020) 52--64.

\bibitem{zerovis}  D. Dereniowski, D. Dyer, R.M. Tifenbach, B. Yang, The complexity of zero-visibility cops and robber, Theoret. Comput. Sci. 607 (2015) 135–148.

\bibitem{zerovis2} D. Dereniowski, D. Dyer, R.M. Tifenbach, B. Yang, Zero-visibility cops and robber and the pathwidth of a graph, J. Comb. Optim. 29 (3) (2015)
541–564.

\bibitem{tosic} R.\ To\v{s}i\'{c}, On Cops and Robber Game, 23rd edition, Studia Scientiarum Mathematicarum Hungarcia, 225--229, Yugoslavia, 1988. 
\end{thebibliography}
\end{document}